\documentclass[11pt,a4paper]{article}

\usepackage{amsmath}
\usepackage{amssymb}
\usepackage{stmaryrd}
\makeatletter

\usepackage[margin=1in]{geometry}
\usepackage{hyperref}
\usepackage{mathtools}
\usepackage{amsthm}
\usepackage{upquote}
\usepackage{amsfonts}
\usepackage{amssymb}\usepackage{amsmath}
\usepackage{amsfonts}
\usepackage{amssymb}
\usepackage{bbm}
\usepackage{enumitem}
\usepackage{tikz}
\usetikzlibrary{decorations.markings}
\usepackage{tikz-cd}
\usepackage{authblk}
\usetikzlibrary{babel}

\newtheorem{theorem}{Theorem}[section]
\newtheorem{lemma}[theorem]{Lemma}
\newtheorem{proposition}[theorem]{Proposition}
\newtheorem{corollary}[theorem]{Corollary}

\newtheorem{definition}{Definition}[section]
\newtheorem{example}{Example}[section]
\theoremstyle{remark}
\newtheorem{remark}[theorem]{Remark}

\AtEndDocument{
\textsc{Mathematical and Statistical Sciences, University of Alberta, Edmonton, Canada} \\
\textsc{Email:} \texttt{primozic@ualberta.ca}}

\begin{document}
\title{Motivic cohomology and infinitesimal group schemes}
\author{Eric Primozic}

\maketitle
\setcounter{secnumdepth}{1}

\renewcommand{\thefootnote}{\fnsymbol{footnote}} 
\renewcommand{\thefootnote}{\arabic{footnote}} 

\begin{abstract}
For $k$ a perfect field of characteristic $p>0$ and $G/k$ a split reductive group with $p$ a non-torsion prime for $G,$ we compute the mod $p$ motivic cohomology of the geometric classifying space $BG_{(r)}$, where $G_{(r)}$ is the $r$th Frobenius kernel of $G.$ Our main tool is a motivic version of the Eilenberg-Moore spectral sequence, due to Krishna. 

\indent For a flat affine group scheme $G/k$ of finite type, we define a cycle class map from the mod $p$ motivic cohomology of the classifying space $BG$ to the mod $p$ \'etale motivic cohomology of the classifying stack $\mathcal{B}G.$ This also gives a cycle class map into the Hodge cohomology of $\mathcal{B}G.$ We study the cycle class map for some examples, including Frobenius kernels.
\end{abstract}
\section*{Introduction}
\indent Let $k$ be a field of characteristic $p>0$ and let $G$ be a split reductive group over $k.$ The relative Frobenius morphism $F:G \to G^{(1)}$ is a group homomorphism and we can consider the $r$th Frobenius kernel of $G$: $$G_{(r)}=\textup{ker}(F^{r}: G \to G^{(r)})$$ for all $r\geq1.$ The group schemes $G_{(r)}$ have been studied extensively in representation theory. For $G$ simple and simply connected, Friedlander and Parshall proved that \begin{equation} \label{cohfirstfrobkernel}H^{*}(G_{(1)}, k) \cong k[\mathcal{N}]\end{equation} where $\mathcal{N} \subset \mathfrak{g}=\textup{Lie}(G)$ is the nilpotent cone of $\mathfrak{g}$, with some bounds on the prime $p$ that depend on the Coxeter number $h(G)$ of $G$ \cite[Theorem 1.4]{FriPar}. Andersen and Jantzen improved the bound on $p$ to show that the isomorphism \ref{cohfirstfrobkernel} holds if $p>h(G)$ \cite{AndJan}. More generally, for $r \geq 1,$ Suslin, Friedlander, and Bendel proved that there is a homeomorphism of topological spaces \begin{equation} \label{cohallfrobkern} \textup{Spec}(k[V_{r}(G)]) \to \textup{Spec}(H^{*}(G_{(r)}, k))\end{equation} where $V_{r}(G)$ is an affine $k$-scheme parameterizing  morphisms $\mathbb{G}_{a(r)} \to G$ \cite{SusFriBen, SusFriBen2}.

\indent In this paper, we consider the analogous problem of computing the motivic cohomology of the geometric classifying space $BG_{(r)}.$ By a classifying space $BH$ for a flat affine group scheme $H$ of finite type over a field, we mean the construction given by Totaro in \cite{TotChow} and Morel-Voevodsky in \cite{MorVoe} where $BH$ is approximated by smooth fppf quotients $U/H.$ Thus, we are studying characteristic classes for $G_{(r)}$-torsors that are locally trivial in the fppf topology. 

\indent Using that the Frobenius $F^{r}$ is trivial on $G_{(r)},$ we show in Proposition \ref{proptorsion} that all nontrivial cohomology classes in $H^{*}(BG_{(r)}, \mathbb{Z}(*))$ are $p$-power torsion and that torsion varies with weight. Hence, it is natural to consider motivic cohomology with $\mathbb{Z}/p$-coefficients for infinitesimal group schemes. In recent work, the mod $p$ motivic cohomology of $B\mu_{p}$ (where $\mu_{p}$ denotes the group scheme of the $p$th roots of unity over $k$) has been used to determine part of the dual Steenrod algebra and define Steenrod operations on motivic cohomology in equal characteristic \cite{FrankSpi, Pri}. Here, the mod $p$ motivic cohomology of the geometric classifying space $B\mu_{p}$ is nontrivial, while the motivic cohomology of $B\mathbb{Z}/p$ is trivial. (More generally, the motivic cohomology of $BG$ is trivial for $G$ a finite \'etale $p$-group scheme over $k$ \cite[Proposition 3.3]{MorVoe}.) The motivic cohomology of $B\mu_{p}$ is given by $$H^{*}(B\mu_{p}, \mathbb{F}_{p}(*))\cong H^{*}(k, \mathbb{F}_{p}(*))[u,v]/(u^{2})$$ where $v \in H^{2}(B\mu_{p}, \mathbb{F}_{p}(1))$ is the first Chern class under the map $B\mu_{p} \to B\mathbb{G}_{m}$ and $u \in H^{1}(B\mu_{p}, \mathbb{F}_{p}(1))$ is related to $v$ by the Bockstein $\beta(u)=v$ \cite[Theorem 6.10]{Voe}. Hence, we have $$H^{*}(B\mu_{p}, \mathbb{F}_{p}(*)) \cong H^{*}(B\mathbb{G}_{m}, \mathbb{F}_{p}(*)) \otimes _{H^{*}(k,\mathbb{F}_{p}(*))} H^{*}(\mathbb{G}_{m}, \mathbb{F}_{p}(*)) .$$ We generalize this result with the following theorem. Let $r_{1}, \ldots, r_{m}$ denote the \textit{fundamental degrees} of $G$, allowing repetitions so that $CH^{*}(BG) \otimes \mathbb{Q}$ is a polynomial ring on $m$ generators.

\begin{theorem}[Theorem \ref{theoremMotcohBG_r} and Corollary \ref{bocksteiniso}]  Let $G$ be a split reductive group over a perfect field $k$ of characteristic $p>0$ with $p$ a non-torsion prime of $G.$ 
\begin{enumerate}
\item There is an isomorphism $$H^{*}(BG_{(r)}, \mathbb{F}_{p}(*)) \cong CH^{*}(BG)/p  \otimes H^{*}(G, \mathbb{F}_{p}(*))$$
$$\cong \mathbb{F}_{p}[y_{1}, \ldots ,y_{m}] \otimes \bigwedge \nolimits ^{*}(g_{1}, \ldots, g_{m})$$ of graded $CH^{*}(BG)/p$-modules where $y_{i} \in H^{2r_{i}}(BG_{(r)}, \mathbb{F}_{p}(r_{i})) $ and $g_{i} \in H^{2r_{i}-1}(BG_{(r)}, \mathbb{F}_{p}(r_{i}))$ for all $i.$ This is an isomorphism of graded rings if $p>2.$
\item For each $i \leq m,$ $$\delta(g_{i})=p^{rr_{i}-1}y_{i} \in CH^{*}(BG_{(r)})_{(p)}$$ where $\delta$ is the integral Bockstein homomorphism.
\end{enumerate}
\end{theorem}

\indent To prove this theorem, we use a motivic analogue of the Eilenberg-Moore spectral sequence, constructed by Krishna in \cite{Kri}. Krishna's spectral sequence requires the torsion index (see section \ref{eilenbergmooresection} for a definition of the torsion index) of the split reductive group being considered to be inverted. We first use the Eilenberg-Moore spectral sequence for the fibration $$G \to \textup{Spec}(k) \to BG$$ to compute the mod $p$ motivic cohomology of $G$ where $p$ is a non-torsion prime for $G.$ We note that Biglari previously computed the geometric motives of split reductive groups with rational coefficients \cite{Big}, and Pushin computed the integral motivic cohomology of $GL_{n}$ over base fields of arbitrary characteristic \cite{Pus}. After computing the mod $p$ motivic cohomology of $G,$ we run the Eilenberg-Moore spectral sequence on the fibration $$G \cong G/G_{(r)} \to BG_{(r)} \to BG$$ to compute the mod $p$ motivic cohomology of $BG_{(r)}.$  

\indent Next, we relate our calculations to \'etale motivic cohomology (also known as Lichtenbaum cohomology). Along with motivic cohomology, \'etale motivic cohomology has been useful in determining cohomological invariants of algebraic groups. (See \cite{BliMer}, for example.) We note that extra care is needed in defining \'etale motivic cohomology and a cycle class map for objects more general than schemes since \'etale motivic cohomology for schemes is not homotopy invariant in positive characteristic. In section \ref{sectioncycleclass}, for $G$ a flat affine group scheme of finite type over $k,$ we define a cycle class map from the mod $p^{r}$ motivic cohomology of $BG$ to the mod $p^{r}$ \'etale motivic cohomology of the classifying stack $\mathcal{B}G:$ $$H^{*}(BG, \mathbb{Z}/p^{r}(*))  \to H^{*}_{\textup{\'et}}(\mathcal{B}G, \mathbb{Z}/p^{r}(*)).$$ We also obtain a cycle map from the mod $p$ motivic cohomology of the classifying space $BG$ to the Hodge cohomology of the classifying stack $\mathcal{B}G:$ $$H^{*}(BG, \mathbb{Z}/p(*)) \otimes k \to H^{*}_{\textup{H}}(\mathcal{B}G/k).$$ We study the cycle class map for some examples. For $G$ a split reductive group over $k$ with $p$ a non-torsion prime for $G,$ we show that the cycle class map from $CH^{*}(BG_{r})/p$  into mod $p$ \'etale motivic cohomology and Hodge cohomology is injective for all Frobenius kernels $G_{(r)}.$ 



\section*{Acknowledgments} I thank Burt Totaro for his suggestions.
\section{Motivic cohomology}
Let $F$ be a field and let $X$ be a smooth scheme of finite type over $F.$ Motivic cohomology groups of $X$ are bigraded groups $H^{i}(X, A(j))$ where $A$ is a coefficient ring and $j \geq 0.$ Here, for $x \in H^{i}(X, A(j)),$ we let $\textup{deg}(x) \coloneqq i$ denote the degree of $x$ and $w(x) \coloneqq j$ is the weight. Motivic cohomology groups are isomorphic to Bloch's higher Chow groups \cite[Corollary 2]{Voe2}.

\begin{theorem}  Let $X$ be a smooth scheme of finite type over a field $F.$ Let $i \in \mathbb{Z}$ and let $j \geq 0.$ There is a natural isomorphism $$H^{i}(X,\mathbb{Z}(j)) \cong CH^{j}(X, 2j-i).$$
\end{theorem}
This implies the vanishing $H^{i}(X, \mathbb{Z}(j))=0$ of motivic cohomology if $i>2j.$ Also,  $$H^{i}(X, \mathbb{Z}(j))=0$$ for $i>j+\textup{dim}(X).$  Another type of vanishing that we will need to consider for our calculations is related to the Beilinson-Soul\'e conjecture, which asserts that motivic cohomology vanishes in negative degrees.
\begin{proposition} \label{BeiSou} Let $X$ be smooth and of finite type over a field $k$ of characteristic $p>0.$ Then $H^{i}(X, \mathbb{Z}/p^{r}(j))=0$ for all $i<0,$ $j \geq 0,$ and $r \geq 1.$
\end{proposition}
\begin{proof}
Akhtar proves this in \cite[Corollary 2.7]{Akh}, using a result of Geisser-Levine \cite[Theorem 1.1]{GeiLev}.
\end{proof}

\indent Motivic cohomology is contravariant with respect to morphisms between smooth schemes over a field and covariant with respect to proper morphisms. We will also need to use that arbitrary flat morphisms induce pullback maps on motivic cohomology \cite[A.4]{Gei2}. An important example for us will be the absolute Frobenius $F_{X}: X \to X$ (defined as being the $p$th power map on the structure sheaf) for $X/k$ a smooth scheme with $k$ a field of characteristic $p>0.$ Here, $F_{X}$ is not compatible with the map $X \to \textup{Spec}(k)$ unless $k=\mathbb{F}_{p}.$ For a Noetherian ring $R$ of characteristic $p>0$ with $p$ a prime, Kunz proved that the Frobenius $F_{R}: R \to R$ is flat if and only if $R$ is regular \cite{Kun}.

\indent For every $m \in \mathbb{N},$ we let $$\beta:H^{*}(-,\mathbb{Z}/m(*)) \to H^{*+1}(-,\mathbb{Z}/m(*)),$$ $$\delta:H^{*}(-,\mathbb{Z}/m(*)) \to H^{*+1}(-,\mathbb{Z}(*))$$ denote the Bockstein homomorphisms.


\section{Classifying spaces}
 We recall some facts about quotients of algebraic groups (found in \cite[I.5]{Jan}, for example). Let $G$ be a flat affine group scheme of finite type over a field $k.$ Let $X/k$ be a quasi-projective scheme such that $G$ acts freely on $X.$ A morphism $f:X \to Y$ of schemes of finite type over $k$ is called a principal $G$-bundle (or $G$-torsor) if $f$ is faithfully flat, $f$ is constant on the $G$-orbits of $X$, and the map $G \times _{k} X \to X \times_{Y} X$ induced by $f$ is an isomorphism. Here, we write $Y=X/G$ and refer to $Y$ as being the quotient of $X$ by the action of $G.$ If $X$ is smooth over $k,$ then $Y$ is also smooth over $k$ \cite[Tag 02KJ]{Sta}. 
 
 \indent If we further assume that $G$ is a finite group scheme (meaning that $\textup{dim}_{k}(k[G])<\infty$), then all quotients by free actions of $G$ exist as schemes \cite[Tag 07S5]{Sta}. Here, for a finite group scheme $G$ acting on $X$ as above, $\mathcal{O}_{X/G}$ can be identified with the sheaf of invariants $\mathcal{O}_{X}^{G}$ \cite[Lemma 2.5 and Remark 2.6]{Bri}. Furthermore, $X/G$ is quasi-projective as well, even when $G$ is infinitesimal \cite[Lemma 2.5]{Bri}.
 
 \indent The classifying space $BG$ was constructed by Totaro \cite{TotChow} and Morel-Voevodsky \cite{MorVoe} (geometric classifying space) by approximating $BG$ by smooth quotients $U/G.$ Let $V$ be a representation of $G$ over $k$ such that $G$ acts freely on $U \coloneqq V \setminus S$ for some closed $G$-invariant $S \subset V$ of codimension $> r$, and such that the quotient $U/G$ exists as a scheme over $k.$ (Such $V, U,$ and $S$ exist for arbitrarily large $r$ by \cite[Remark 1.4]{TotChow}.) Then $H^{i}(BG, \mathbb{Z}(j)) \coloneqq H^{i}(U/G, \mathbb{Z}(j))$ for $j< r.$ The definition of $BG$ is independent of choices \cite[Theorem 1.1]{TotChow}.  More generally, for an action of $G$ on a smooth scheme $X/k,$ Edidin-Graham defined equivariant motivic cohomology $H^{*}_{G}(X,\mathbb{Z}(*))$ as the motivic cohomology of $X//G=(X \times EG)/G,$ which is defined by taking approximations $(X \times U)/G$ for $U=V\setminus S, V$ as above with the codimension of $S$ sufficiently large \cite{EdiGra}.
 
 \begin{example} Let $G \to H$ be an inclusion of algebraic groups over $k.$ Then $$BG \cong (H/G)//H=(EH \times H/G)/H.$$ We have a fibration $$H/G \to BG \to BH.$$
 \end{example}
 
\indent Now we recall the definition of Frobenius kernels of affine group schemes. See \cite[I.9]{Jan} for a more detailed discussion. Assume that $k$ is a field of characteristic $p>0.$ For a scheme $S$ over $k,$ we let $F_{S}:S \to S$ denote the absolute Frobenius morphism, and define $S^{(1)}$ by the following fiber product.
\[
\begin{tikzcd}
S^{(1)} \arrow[r, ""] \arrow[d,""] & S \arrow[d, ""] \\
\textup{Spec}(k) \arrow[r, "F_{k}"] & \textup{Spec}(k)
\end{tikzcd}
\]
The Frobenius $F_{S}$ on $S$ then induces the relative Frobenius $F_{S/k}: S \to S^{(1)}.$ Specializing to the case where $S=G$ with $G/k$ an affine group scheme, $G^{(1)}$ inherits a group structure from $G$ such that $F_{G/k}:G \to G^{(1)}$ is a homomorphism of group schemes. Iterating, we get a homomorphism of group schemes $F^{r}_{G/k}:G \to G^{(r)}$ for all $r>0.$ For $r>0,$ the $r$th Frobenius kernel of $G$ is defined by $G_{(r)} \coloneqq \textup{ker}(F^{r}_{G/k}).$ For $G$ split reductive, $G$ is defined over $\mathbb{F}_{p}$ and so we have $G^{(r)} \cong G$ for all $r \geq 0.$ 

\indent The following proposition shows that all nontrivial elements in $H^{*}(BG_{(r)}, \mathbb{Z}(*))$ are $p$-power torsion. The second part of this proposition (suggested to us by B. Totaro) gives a better estimate on the torsion of cohomology in large weights. Eventually, we will consider the action of Bockstein homomorphisms on motivic cohomology with $\mathbb{Z}/p^{N}$-coefficients (so precise information on torsion will be needed).

\begin{proposition} \label{proptorsion} Let $G$ be a split reductive group over a field $k$ of characteristic $p>0$ with identity element $\textup{Spec}(k) \to G,$ and let $r>0.$ Let $\pi: B\textup{Spec}(k) \to BG$ denote the corresponding map of classifying spaces. Let $x \in H^{i}(BG_{(r)}, \mathbb{Z}(j))$ for some $i,j \geq 0$ and assume that $\pi^{*}x=0 \in H^{i}(k, \mathbb{Z}(j)).$
\begin{enumerate}

\item We have $p^{rj}x=0$ in $H^{i}(BG_{(r)}, \mathbb{Z}(j)).$
\item We have $p^{r\textup{dim}(G)}x=0$ in $H^{i}(BG_{(r)}, \mathbb{Z}(j)).$
\end{enumerate}
\end{proposition}

\begin{proof}
First, we describe the action of the Fronenius morphism on motivic cohomology. Let $X/k$ be a smooth variety. Consider the flat pullback map $F_{X}^{*}$ defined on the motivic cohomology of $X.$ We wish to prove that $F_{X}^{*}(x)=p^{j}x$ for $x \in H^{i}(X, \mathbb{Z}(j)).$  If $k=\mathbb{F}_{p},$ this is a known result \cite[Theorem 4.6]{Gei}. The case for a general field $k$ of characteristic $p$ then follows since $X$ is essentially smooth over $\mathbb{F}_{p},$ meaning that $X$ is a cofiltered limit of smooth schemes over $\mathbb{F}_{p}.$  (See \cite[\S 2]{HKO} for results on essentially smooth schemes and motivic cohomology.)

\indent Let $V$ be a representation of $G_{(r)}$ such that $G_{(r)}$ acts freely on some open $U=V\setminus S$ where $S \subset V$ is a closed $G_{(r)}$-subset of codimension $>j$ such that $U/G_{(r)}$ exists as a scheme. Let $\pi:U \to U/G_{(r)}$ denote the quotient map. As $F^{r}_{G_{(r)}}:G_{(r)} \to G_{(r)}$ is trivial, there exists a map $\psi:U/G_{(r)} \to U$ such that the following triangle commutes.
\[
\begin{tikzcd}
U/G_{(r)} \arrow[r, "\psi"] \arrow[dr, "F^{r}_{U/G_{(r)}}"] & U \arrow[d, "\pi"] \\
& U/G_{(r)}
\end{tikzcd}
\]
Note that $\psi$ is a flat morphism since $\pi$ is faithfully flat and $F^{r}_{U/G_{(r)}}$ is flat \cite[Tag 06NB]{Sta}.

\indent From the localization sequence for motivic cohomology, we have $$H^{i}(U, \mathbb{Z}(j)) \cong H^{i}(V, \mathbb{Z}(j)) \cong  H^{i}(k, \mathbb{Z}(j)).$$ Then $$0=(F^{r}_{U/G_{(r)}})^{*}(x)=\psi^{*}(\pi^{*}(x)) =p^{rj}x.$$ This proves the first part.

\indent For the second part, we use that $\pi$ is a finite morphism of degree $\textup{dim}_{k}(k[G_{(r)}])=p^{r\textup{dim}(G)}$ \cite[I.9.6(2)]{Jan}. Then $$0=\pi_{*}\pi^{*}x=p^{r\textup{dim}(G)}x.$$

\end{proof}
\begin{remark} \label{remarkField} For $k$ a field of characteristic $p>0,$ Geisser-Levine proved that $H^{i}(k, \mathbb{Z}(j))$ is uniquely $p$-divisible for $i \neq j$ \cite[Theorem 1.1]{GeiLev}. Hence, if $k$ is also perfect, $H^{i}(k, \mathbb{Z}(j))$ is $p$-divisible for $j \neq 0.$ 
\end{remark}
Let $k$ be a perfect field of characteristic $p>0.$ For $n \in \mathbb{N},$ we let $\mu_{n}$ denote the group scheme of $n$th roots of unity in $k.$
\begin{proposition} \label{CohBmu}
Let $n \in \mathbb{N}.$ There is an isomorphism of rings $$H^{*}(B\mu_{p^{n}}, \mathbb{Z}/p(*)) \cong \mathbb{F}_{p}[u,v]/(u^{2})$$ where $v \in H^{2}(B\mu_{p^{n}}, \mathbb{Z}/p(1))$ is induced by the inclusion $\mu_{p^{n}} \to \mathbb{G}_{m}$ and $u \in H^{1}(B\mu_{p^{n}}, \mathbb{Z}/p(1)).$
\end{proposition}
\begin{proof}
This result is well-known. Let $DM(k, \mathbb{Z}/p)$ denote Voevodsky's triangulated category of motives with coefficients in $\mathbb{Z}/p$ \cite{MVW}. Throughout the proof, we consider motives with $\mathbb{Z}/p$-coefficients. The motive $M(B\mu_{p^{n}})$ of $B\mu_{p^{n}}$ is isomorphic to $\mathcal{O}_{\mathbb{P}^{\infty}}(p^{n})- z_{\mathbb{P}^{\infty}}$ where $z_{\mathbb{P}^{\infty}} \subset \mathcal{O}_{\mathbb{P}^{\infty}}(p^{n})$ is the zero section of the line bundle $\mathcal{O}_{\mathbb{P}^{\infty}}(p^{n}) \to \mathbb{P}^{\infty}.$ We have $$M(B\mathbb{G}_{m})=\mathbb{P}^{\infty} \cong \bigoplus^{\infty}_{i=0} \mathbb{Z}/p(i)[2i]$$ in $DM(k, \mathbb{Z}/p).$ The Gysin triangle 
\[
\begin{tikzcd}
 M(B\mu_{p^{n}}) \arrow[r, ""] &  M(B\mathbb{G}_{m}) \arrow[r, "p^{n}c_{1}=0"] &M(B\mathbb{G}_{m})(1)[2] \arrow[r, ""]& M(B\mu_{p^{n}})[1]
 \end{tikzcd}
 \]
 splits, giving $$M(B\mu_{p^{n}}) \cong M(B\mathbb{G}_{m}) \oplus M(B\mathbb{G}_{m})(1)[1]$$ in $DM(k, \mathbb{Z}/p).$ To determine the product structure, we use that $H^{2}(B\mu_{p^{n}}, \mathbb{Z}(1)) \cong \mathbb{Z}/p^{n}$ so that the Bockstein $\delta:H^{1}(B\mu_{p^{n}}, \mathbb{Z}/p(1)) \to H^{2}(B\mu_{p^{n}},\mathbb{Z}(1))$ maps a generator $u \in H^{1}(B\mu_{p^{n}}, \mathbb{Z}/p(1))$ to a nonzero element of $H^{2}(B\mu_{p^{n}},\mathbb{Z}(1)).$
 \end{proof}
\section{Eilenberg-Moore spectral sequence} \label{eilenbergmooresection}

In this section, we recall the motivic Eilenberg-Moore spectral sequence of Krishna. We follow \cite[Chapter 16]{Tot2} where a construction of the Eilenberg-Moore spectral sequence is given (with cohomological grading) and where the notation is similar to ours. Let $G$ be a split reductive group over a field $k$ with Borel subgroup $B$, split torus $T\subset B,$ and Weyl group $W.$ Set $N=\textup{dim}(G/B)$ and $X^{*}(T)=\textup{Hom}(T, \mathbb{G}_{m}).$ There is a natural map $X^{*}(T) \to CH^{1}(G/B).$  The torsion index $t_{G}\in \mathbb{N}$ of $G$ (defined by Grothendieck in \cite{Gro}) is the order of the cokernel of the homomorphism $S^{N}(X^{*}(T)) \to CH^{N}(G/B) \cong \mathbb{Z}.$ Inverting the torsion index, the map $BT \to BG$ gives an isomorphism \begin{equation} \label{ChowTorsion}CH^{*}(BG)[t_{G}^{-1}] \cong CH^{*}(BT)^{W}[t_{G}^{-1}]\cong \mathbb{Z}[t_{G}^{-1}][y_{1}, \ldots, y_{m}]\end{equation} where $\textup{deg}(y_{1}), \ldots \textup{deg}(y_{m})$ are the fundamental degrees of $G$ \cite[Theorem 1.3]{Tot1}, \cite[Th\'eor\`eme]{Dem}. A prime $l$ is a said to be a torsion prime of $G$ if $l \mid t_{G}$ and is called a non-torsion prime otherwise.

\indent Recall that an algebraic group $G$ over a field $k$ is called special if all $G$-torsors are Zariski locally trivial. In this case, $t_{G}=1.$ For example, the groups $GL(n), SL(n),$ and $Sp(2n)$ are all special and hence have torsion index equal to $1.$ More generally, for a simply connected simple group $G,$ $p$ is a torsion prime for $G$ if $p=2$ and $G$ is not of type $A_{n}$  or $C_{n};$ if $p=3$ and $G$ is not of type $A_{n}, B_{n}, C_{n}, D_{n},$ or $G_{2};$ and if $p=5$  and $G$ is of type $E_{8}$ \cite[Corollary 1.13]{Ste}.

\indent Let $R$ be a bigraded ring and let $M,N$ be bigraded $R$-modules. For $i \geq0$ and $q,j \in \mathbb{Z},$ let $\textup{Tor}_{i,q,j}^{R}(M,N)$ denote the $(q,j)$th bigraded piece of $\textup{Tor}_{i}^{R}(M,N).$ 
\begin{definition} For independent variables $x_{1} \ldots,x_{m}$ and a ring $R,$ we let $\bigwedge_{R}^{*}(x_{1}, \ldots, x_{m})$ denote the free $R$-module with basis given by $dx_{i_{1}} \wedge \cdots \wedge dx_{i_{j}}$ where $1 \leq i_{1}< i_{2} <\cdots < i_{j} \leq m.$ 
\end{definition}
\indent The notation  $\Delta_{R}(x_{1} ,\ldots, x_{m})$ is also used in the literature to denote the same thing, but we will be considering product structures (given by wedge products) so the former notation seems better. Also, we will suppress the ring $R$ from our notation. 
\begin{theorem} \label{EilenbergMoore} Let $G$ be a split reductive group over a field $k,$ acting on a smooth scheme $X/k.$ Let $T \subset G$ be a split torus with $CH^{*}(BT) \cong \mathbb{Z}[t_{1}, \ldots, t_{m}].$ Let $j \geq 0.$
\begin{enumerate} 
\item There exists a left half-plane convergent spectral sequence $$E_{1}^{p,q}=H^{2p+q}_{T}(X, \mathbb{Z}(j+p)) ^{\oplus \binom{m}{-p}} \Rightarrow H^{p+q}(X, \mathbb{Z}(j)).$$
\item There exists a left half-plane convergent spectral sequence \[ \label{secondEMSS}E_{2}^{p,q}=\textup{Tor}^{CH^{*}(BG)}_{-p,q,j}(\mathbb{Z}, H^{*}_{G}(X, \mathbb{Z}(*)))[t_{G}^{-1}] \Rightarrow H^{p+q}(X, \mathbb{Z}(j))[t_{G}^{-1}].\]
\item Let $l$ be a non-torsion prime for $G$ and let $n \in \mathbb{N}.$ There exists a second-quadrant convergent spectral sequence $$E_{2}^{p,q}=\textup{Tor}^{CH^{*}(BG)/l^{n}}_{-p,q,j}(\mathbb{Z}/l^{n}\mathbb{Z}, H^{*}_{G}(X, \mathbb{Z}/l^{n}\mathbb{Z}(*))) \Rightarrow H^{p+q}(X, \mathbb{Z}/l^{n}\mathbb{Z}(j)).$$
\end{enumerate}
\end{theorem}
\begin{proof} The first spectral sequence is \cite[Corollary 16.3]{Tot2} and the $E_{1}$ page is given by tensoring the Koszul resolution of $\mathbb{Z}$ as a $CH^{*}(BT)$-module with $H^{*}_{T}(X, \mathbb{Z}(*)):$ $$H^{*}_{T}(X, \mathbb{Z}(*)) \otimes \bigwedge \nolimits^{*}(t_{1}, \ldots ,t_{m})$$ where the differentials are given by the usual antiderivation on $\bigwedge \nolimits^{*}(t_{1}, \ldots ,t_{m}).$ The $E_{2}$ page of the first spectral sequence is hence given by $$E_{2}^{p,q}=\textup{Tor}^{CH^{*}(BT)}_{-p,q,j}(\mathbb{Z}, H^{*}_{T}(X, \mathbb{Z}(*))).$$ The second and third spectral sequences can then be obtained from the first after inverting $t_{G}:$ \begin{equation} \label{toriso}\textup{Tor}^{CH^{*}(BT)}_{-p,q,j}(\mathbb{Z}, H^{*}_{T}(X, \mathbb{Z}(*)))[t_{G}^{-1}] \cong \textup{Tor}^{CH^{*}(BG)}_{-p,q,j}(\mathbb{Z}, H^{*}_{G}(X, \mathbb{Z}(*)))[t_{G}^{-1}].\end{equation} This isomorphism uses that $$H^{*}_{T}(X, \mathbb{Z}(*))[t_{G}^{-1}] \cong H^{*}_{G}(X, \mathbb{Z}(*)) \otimes _{CH^{*}(BG)} CH^{*}(BT)[t_{G}^{-1}]$$ \cite[Theorem 16.1]{Tot2} along with flat base change for Tor \cite[Proposition 3.2.9]{Wei}. We can say that the third spectral sequence is a second-quadrant spectral sequence since negative degree motivic cohomology is guaranteed to vanish in this case by Proposition \ref{BeiSou}.
\end{proof}

\section{Motivic cohomology of $G$}
Let $G$ denote a split reductive group over an arbitrary base field with a Borel subgroup $B \subset G.$ The Chow ring of $G$ is known (over arbitrary fields), from Grothendieck \cite{Gro}:
$$CH^{*}(G) \cong CH^{*}(G/B)/I$$ where $I \subset CH^{*}(G/B)$ is the ideal generated by the image of $CH^{*}(BT) \to CH^{*}(G/B).$ (This can also be deduced from the integral Eilenberg-Moore spectral sequence.) From the definition of the torsion index, we then get the following lemma.
\begin{lemma} \label{ChowG} Let $t_{G} \in \mathbb{N}$ denote the torsion index of $G.$ Then $CH^{m}(G)[t_{G}^{-1}]=0$ for $m>0.$
\end{lemma}

\indent For the rest of this section, our goal is to compute $H^{*}(G, \mathbb{Z}/p(*))$ where $G$ is a split reductive group over a perfect field $k$ with $\textup{char}(k)=p>0$ a non-torsion prime of $G.$ Our proof is similar to arguments given by Biglari in \cite{Big}, where rational motives of split reductive groups are computed (a similar proof can be used in topology as well). We also note that Pushin computed the integral motivic cohomology of $GL_{n}$ over a field $F$ for all $n,$ showing that it is almost an exterior algebra: $$H^{*}(GL_{n}, \mathbb{Z}(*)) \cong H^{*}(F, \mathbb{Z}(*)) \otimes \bigwedge \nolimits^{*}(c_{1,1}, c_{1,2}, \ldots, c_{1,n})$$ where $c_{1,i} \in H^{2i-1}(G, \mathbb{Z}(i))$ and $c_{1,i}^{2}=\rho c_{1,2i-1}$ for all $i$ where $\rho \in H^{1}(F, \mathbb{Z}(1))\cong F^{\times}$ is the class of $-1$ \cite[Proposition 2]{Pus}.

\indent From our assumptions on $k,$ we have $H^{*}(k, \mathbb{Z}/p(*)) \cong \mathbb{Z}/p$ (see remark \ref{remarkField}).
\begin{proposition} \label{corCohG} Let $G$ be a split reductive group over a field $k.$ Let $r_{1} ,\ldots , r_{m}$ denote the fundamental degrees of $G.$
Assume that $k$ is perfect and $\textup{char}(k)=p>0$ is a non-torsion prime of $G.$ Then \begin{equation} \label{cohG} H^{*}(G, \mathbb{Z}/p(*)) \cong \bigwedge \nolimits^{*}(x_{1}, \ldots , x_{m}) \end{equation} as graded rings where $x_{i} \in H^{2r_{i}-1}(G, \mathbb{Z}/p(r_{i}))$ for all $i.$
\end{proposition}
\begin{proof}
 As mentioned above, $CH^{*}(BG)/p \cong \mathbb{F}_{p}[y_{1}, \ldots, y_{m}]$ where $$y_{i} \in CH^{r_{i}}(BG)/p\cong H^{2r_{i}}(BG, \mathbb{Z}/p(r_{i})).$$ We then consider the Eilenberg-Moore spectral sequence corresponding to the fibration $G \to \textup{Spec}(k) \to BG$:  \begin{equation} \label{EMSSG} E_{2}^{l,q}=\textup{Tor}^{CH^{*}(BG)/p}_{-l,q,j}(\mathbb{Z}/p, H^{*}(\textup{Spec}(k), \mathbb{Z}/p(*))) \Rightarrow H^{l+q}(G, \mathbb{Z}/p(j)).\end{equation} The Tor groups can be computed here by tensoring the Koszul resolution of $\mathbb{Z}/p$ as a module over the polynomial ring $CH^{*}(BG)/p$ with $H^{*}(\textup{Spec}(k), \mathbb{Z}/p(*))\cong \mathbb{Z}/p$ \cite[Corollary 4.5.5]{Wei}: $$0 \to (\mathbb{Z}/p)^{\binom{m}{m}}  \to \cdots \to (\mathbb{Z}/p)^{\binom{m}{1}} \to (\mathbb{Z}/p) \to 0.$$ Here, the differentials in this sequence are all $0,$ and $$y_{i_{1}} \wedge \cdots \wedge y_{i_{j}} \in  \bigwedge\nolimits^{j}(y_{1}, \ldots, y_{m}) \cong (\mathbb{Z}/p)^{\binom{m}{j}}$$ has bidegree $(2(r_{i_{1}}+ \cdots + r_{i_{j}}),r_{i_{1}}+ \cdots + r_{i_{j}}) $ for all $j \leq m$ and $i_{1}, \ldots, i_{j} \leq m.$ It follows that the spectral sequence \ref{EMSSG} degenerates, giving the isomorphism \ref{cohG} as vector spaces.

\indent To show that \ref{cohG} respects multiplication, we use that $H^{*}(G, \mathbb{Z}/p(*))$ has the structure of a Hopf algebra. See \cite[IV.2 and VII.1]{MimTod} for some relevant facts about Hopf algebras. The motive of $G$ is mixed Tate \cite[Proposition 4.2]{Big}. (This uses the Bruhat decomposition of $G.$) It follows that the K\"{u}nneth spectral sequence \cite[Theorem 8.6]{DugIs}, \cite[Theorem 4.5]{Jos} degenerates, giving $$H^{*}(G\times G, \mathbb{Z}/p(*)) \cong H^{*}(G, \mathbb{Z}/p(*)) \otimes_{\mathbb{Z}/p} H^{*}(G, \mathbb{Z}/p(*)).$$  We then define a comultiplication map $$\phi: H^{*}(G, \mathbb{Z}/p(*)) \to H^{*}(G, \mathbb{Z}/p(*)) \otimes_{\mathbb{Z}/p} H^{*}(G, \mathbb{Z}/p(*))$$ using the multiplication map $G \times G \to G.$ We then view $H^{*}(G, \mathbb{Z}/p(*))$ as being a Hopf algebra over $\mathbb{F}_{p}$, graded by $-(\textup{deg}(*)-w(*)).$  From the classification of graded, connected, and commutative finite quasi-Hopf algebras over $\mathbb{F}_{p},$ we then get that \ref{cohG} is an isomorphism of rings \cite[Theorem VII.1.7]{MimTod}.
\end{proof}

\section{Motivic cohomology of $BG_{(r)}$}

\indent Let $G$ be a split reductive group over a field $k$ of characteristic $p>0$ where $p$ is a non-torsion prime of $G.$ We start by computing $CH^{*}(BG_{(r)})/p.$ Let $T \subset G$ denote a split torus. Our main tool will be the Eilenberg-Moore spectral sequences corresponding to the fibration $$G \cong G/G_{(r)} \to BG_{(r)} \to BG:$$ \begin{equation} \label{EMSSG_r}E_{2}^{l,q}=\textup{Tor}^{CH^{*}(BG)/p}_{-l,q,j}(\mathbb{Z}/p, H^{*}(BG_{(r)}, \mathbb{Z}/p(*))) \Rightarrow H^{l+q}(G, \mathbb{Z}/p(j)).\end{equation} Ultimately, we will show that these spectral sequences degenerate and that the $E_{2}$ pages are concentrated in the $0$th column.

\begin{lemma} \label{lemmachowBG}
Pullback induces an isomorphism $CH^{*}({BG})/p \cong CH^{*}({BG_{(r)}})/p.$
\end{lemma}
\begin{proof}
From the isomorphism \ref{ChowTorsion}, we see that the pullback map $CH^{*}(BG)/p \to CH^{*}(BT)/p$ is injective. We have $T_{(r)} \cong \mu_{p^{r}}^{\textup{rank}(G)}.$ From the calculation of Proposition \ref{CohBmu}, the pullback map $CH^{*}(BT)/p \to CH^{*}(BT_{(r)})/p$ is an isomorphism. It follows that the pullback map $CH^{*}(BG)/p \to CH^{*}(BG_{(r)})/p$ is injective.
\[
\begin{tikzcd}
CH^{*}(BG)/p \arrow[r, ""] \arrow[d, ""] & CH^{*}(BT)/p \arrow[d, ""] \\
CH^{*}(BG_{(r)})/p \arrow[r, ""] & CH^{*}(BT_{(r)})/p 
\end{tikzcd}
\]
\indent From Lemma \ref{ChowG}, we have $H^{2n}(G, \mathbb{Z}/p(n))=0$ for all $n \in \mathbb{N}.$ Hence, in the spectral sequence \ref{EMSSG_r} with $j=n,$  $$0=E_{\infty}^{0,2n} \cong E_{2}^{0,2n}=\textup{Tor}^{CH^{*}(BG)/p}_{0,2n,n}(\mathbb{Z}/p, H^{*}(BG_{(r)}, \mathbb{Z}/p(*)))$$ for all $n \in \mathbb{N}.$ It follows that the map $ CH^{*}(BG)/p \to CH^{*}(BG_{(r)})/p$ is surjective, and hence an isomorphism
\end{proof}
\begin{definition} For a motivic space $X$ (a scheme or classifying space in this paper) and $i \geq 0,$ define $$M_{i}(X) \coloneqq \bigoplus^{\infty}_{n=0} H^{2n-i}(X,\mathbb{Z}/p(n)).$$ Each $M_{i}(X)$ is naturally a $CH^{*}(X)/p$-module.
\end{definition}
\begin{theorem} \label{theoremMotcohBG_r}
Assume that $k$ is a perfect. There is an isomorphism \begin{equation} \label{iso} H^{*}(BG_{(r)}, \mathbb{Z}/p(*)) \cong CH^{*}(BG)/p \otimes_{H^{*}(k,\mathbb{Z}/p(*))}H^{*}(G, \mathbb{Z}/p(*))\end{equation} of graded $CH^{*}(BG)/p$-modules. If $p>2,$ this is also an isomorphism of graded rings.
\end{theorem}
\begin{proof}      
\indent We first show that the pullback map \begin{equation} \label{pullbacktoG} H^{*}(BG_{(r)}, \mathbb{Z}/p(*)) \to H^{*}(G, \mathbb{Z}/p(*))\end{equation} is surjective. From Proposition \ref{corCohG}, it suffices to show that $$M_{1}(BG_{(r)}) \to M_{1}(G)$$ is surjective. Let $n \in \mathbb{N}$ and consider the spectral sequence \ref{EMSSG_r} with weight equal to $n.$
As $M_{0}(BG_{(r)})$ is free as a $CH^{*}(BG)/p$-module and $H^{i}(BG_{(r)}, \mathbb{Z}/p(j))=0$ when $i>2j,$ we have $E_{2}^{l, 2n-1-l}=0$ for $l \leq -1.$ Hence, $$\mathbb{Z}/p \otimes _{CH^{*}(BG)/p}H^{2n-1}(BG_{(r)}, \mathbb{Z}/p(n))=E_{2}^{0,2n-1} \cong E_{\infty}^{0,2n-1} \cong H^{2n-1}(G, \mathbb{Z}/p(n)).$$ Thus, the map \ref{pullbacktoG} is surjective. From now on, we fix a copy $H^{*}(G, \mathbb{Z}/p(*)) \subset H^{*}(BG_{(r)}, \mathbb{Z}/p(*)),$ as a subspace. If $p>2,$ then we can also view $H^{*}(G, \mathbb{Z}/p(*))$ as being a subring of $H^{*}(BG_{(r)}, \mathbb{Z}/p(*))$ since multiplication in motivic cohomomology is graded-commutative with respect to degree.

\indent Notice that the surjection \ref{pullbacktoG} has consequences for the spectral sequence \ref{EMSSG_r}. The $E_{\infty}$ page must be concentrated in the $0$th column with $E_{\infty}^{0,l} \cong H^{l}(G, \mathbb{Z}/p(j))$ for all $l.$ To finish the calculation, we show that the $CH^{*}(BG)/p$-submodule of $H^{*}(BG_{(r)}, \mathbb{Z}/p(*))$ generated by  $H^{*}(G, \mathbb{Z}/p(*)) \subset H^{*}(BG_{(r)}, \mathbb{Z}/p(*))$ is free and equal to $H^{*}(BG_{(r)}, \mathbb{Z}/p(*)).$ We proceed by induction on $i=-(\textup{deg}(*)-w(*)).$ The case $i=0$ is handled above. Now assume that \begin{equation} \label{splitting G} \mathbb{Z}/p \otimes_{CH^{*}(BG)/p}M_{j}(BG_{(r)}) \cong M_{j}(G) \end{equation} and that the $CH^{*}(BG)/p$-submodule generated by $M_{j}(G) \subset M_{j}(BG_{(r)})$ is a free $CH^{*}(BG)/p$-module equal to $M_{j}(BG_{(r)})$ for $j \leq i.$ We use that $$\textup{Tor}^{CH^{*}(BG)/p}_{*}(\mathbb{Z}/p, H^{*}(BG_{(r)}, \mathbb{Z}/p(*))) \cong \bigoplus^{\infty}_{j=0} \textup{Tor}^{CH^{*}(BG)/p}_{*}(\mathbb{Z}/p, M_{j}(BG_{(r)})).$$

\indent For $j \in \mathbb{N}$ and $1 \leq l \leq i,$ we have $E_{2}^{-l, 2j-(i+1)+l}=0$ in \ref{EMSSG_r} by our inductive hypothesis. Hence, there aren't any differentials entering $E_{2}^{0,2j-(i+1)}$ in the spectral sequence \ref{EMSSG_r}. Then $$\mathbb{Z}/p \otimes_{CH^{*}(BG)/p} M_{i+1}(BG_{(r)}) \cong M_{i+1}(G)$$ which implies that the $CH^{*}(BG)/p$-submodule generated by $$M_{i+1}(G) \subset M_{i+1}(BG_{(r)})$$ is equal to $M_{i+1}(BG_{(r)}).$

\indent To finish, we need to show that the $CH^{*}(BG)/p$-submodule generated by $M_{i+1}(G) \subset M_{i+1}(BG_{(r)})$ is free. We proceed by induction on weight. Let $n \in \mathbb{N}$ be the smallest number for which $H^{2n-(i+1)}(BG_{(r)}, \mathbb{Z}/p(n)) \neq 0$ (if no such number exists then we are done). Let $z_{1}, \ldots , z_{m_{i}}$ denote an $\mathbb{F}_{p}$-basis of $M_{i+1}(G) \subset M_{i+1}(BG_{(r)}).$ We have $$H^{2n-(i+1)}(G, \mathbb{Z}/p(n)) \cong H^{2n-(i+1)}(BG_{(r)}, \mathbb{Z}/p(n)).$$  Now assume that $f_{1}z_{1}+\cdots + f_{m_{i}}z_{m_{i}} \neq 0$ for all homogeneous $f_{1}, \ldots, f_{m_{i}} \in \mathbb{F}_{p}[y_{1}, \ldots, y_{m}]\cong CH^{*}(BG)/p$ with at least one $f_{j} \neq 0$ and $|f_{a}z_{a}|=|f_{b}z_{b}| \leq l$ for all nonzero terms.

\indent Suppose that $f_{1}z_{1}+\cdots + f_{m_{i}}z_{m_{i}}=0$ for some homogeneous $f_{1}, \ldots, f_{m_{i}} \in \mathbb{F}_{p}[y_{1}, \ldots, y_{m}]$, not all zero, with $|f_{a}z_{a}|=|f_{b}z_{b}| =l+1$ for all nonzero terms. As the $z_{1}, \ldots, z_{m_{i}}$ form a basis of $M_{i+1}(G) \subset M_{i+1}(BG_{(r)})$ and $f \mapsto 0$ under the pullback $$H^{*}(BG_{(r)}, \mathbb{Z}/p(*)) \to H^{*}(G, \mathbb{Z}/p(*))$$ for all nonconstant $f \in CH^{*}(BG)/p$ with $|f|>0,$ we must have $f_{j} \not \in \mathbb{Z}/p$ for all nonzero $f_{j}.$ We can then write $f_{1}z_{1}+\cdots + f_{m_{i}}z_{m_{i}}=y_{1}z'_{1}+y_{2}z'_{2}+\cdots +y_{m}z'_{m}$ for some $z'_{1}, \ldots, z'_{m} \in  M_{i+1}(BG_{(r)}).$ Here, we require that $$f_{1}\otimes z_{1}+\cdots + f_{m_{i}}\otimes z_{m_{i}}=[y_{1}z'_{1}]+[y_{2}z'_{2}]+\cdots +[y_{m}z'_{m}]$$ in 
$$CH^{*}(BG)/p \otimes_{\mathbb{Z}/p} H^{*}(G, \mathbb{Z}/p(*))$$ where $[y_{a}z'_{a}] \mapsto y_{a}z'_{a}$ for all $a$ under the map $$CH^{*}(BG)/p \otimes_{\mathbb{Z}/p} H^{*}(G, \mathbb{Z}/p(*)) \to H^{*}(BG_{(r)}, \mathbb{Z}/p(*)).$$

\indent Consider the complex $$H^{*}(BG_{(r)}, \mathbb{Z}/p(*)) \otimes_{\mathbb{Z}/p} \bigwedge \nolimits^{*}(y_{1}, \ldots, y_{m})$$ that computes $\textup{Tor}^{CH^{*}(BG)/p}_{*}(\mathbb{Z}/p, H^{*}(BG_{(r)}, \mathbb{Z}/p(*))).$ The differentials $d$ are given by the usual antiderivation on $\bigwedge^{*}(y_{1}, \ldots, y_{m}).$ We have $d(z'_{1} dy_{1}+\cdots+ z_{m}'dy_{m})=y_{1}z'_{1}+y_{2}z'_{2}+\cdots +y_{m}z'_{m}=0$ and $$z'_{1} dy_{1}+\cdots+ z_{m}'dy_{m} \not \in d(H^{*}(BG_{(r)}, \mathbb{Z}/p(*)) \otimes_{\mathbb{Z}/p} \bigwedge \nolimits^{2}(y_{1}, \ldots, y_{m}))$$ by our inductive assumption that there are no $CH^{*}(BG)/p$-relations in $M_{i+1}(BG_{(r)})$ in weights $<l+1.$ Hence, $z'_{1} dy_{1}+\cdots+ z_{m}'dy_{m}$ gives a nonzero element in $$E_{2}^{-1,2(l+1)-(i+1)}=\textup{Tor}^{CH^{*}(BG)/p}_{1, 2(l+1)-(i+1), l+1 }(\mathbb{Z}/p, H^{*}(BG_{(r)}, \mathbb{Z}/p(*)))$$ in the spectral sequence \ref{EMSSG_r} with weight equal to $l+1.$ All differentials leaving $E_{j}^{-1,2(l+1)-(i+1)}$ for $j \geq 2$ must be equal to $0$ since \ref{EMSSG_r} is a second-quadrant spectral sequence. All differentials entering $E_{j}^{-1,2(l+1)-(i+1)}$ for $j \geq 2$ must be equal to $0$ since $E_{2}^{-q, 2(l+1)-(i+1)+q-1}=0$ for all $q>1$ by our inductive assumption that $M_{r}(BG_{(r)})$ is a free $CH^{*}(BG)/p$-module for $r<i+1.$ Then $E_{2}^{-1,2(l+1)-(i+1)} \cong E_{\infty}^{-1,2(l+1)-(i+1)} \neq 0,$ which is a contradiction since we proved above that the $E_{\infty}$ page of $\ref{EMSSG_r}$ is concentrated on the $0$th column. Hence, the $CH^{*}(BG)/p$-submodule generated by $$M_{i+1}(G) \subset M_{i+1}(BG_{(r)})$$ is free. This completes the proof.
\end{proof}
\section{Torsion in motivic cohomology} 
Let $G$ be a split reductive group over a field $k$ with $\textup{char}(k)=p>0$ a non-torsion prime of $G.$ Proposition \ref{proptorsion} gives an upper bound on the order of integral cohomology classes of Frobenius kernels. Fix $n \in \mathbb{N}.$ In this section, we determine the exact orders (additive) of generators in the image of the pullback map $$CH^{*}(BG)[t_{G}^{-1}] \to H^{*}(BG_{(n)}, \mathbb{Z}(*))[t_{G}^{-1}].$$
\indent Let $T \subset G$ denote a split torus with $CH^{*}(BT)\cong \mathbb{Z}[t_{1},\ldots,t_{m}]$ and associated Weyl group $W=N_{G}(T)/T.$ Take $y_{i} \in CH^{*}(BG)[t_{G}^{-1}] \cong CH^{*}(BT)[t_{G}^{-1}]^{W} $ to be generators for $i \leq m$ and $|y_{i}|=r_{i}$ as in \ref{ChowTorsion}. For each $i,$ let $\tilde{y}_{i} \in H^{*}(BG_{(r)}, \mathbb{Z}(*))[t_{G}^{-1}]$ denote the image of $y_{i}.$
\begin{proposition} \label{propositiontorsion}
\begin{enumerate}
\item For $i \leq m,$ the order of $\tilde{y}_{i}$ in  $H^{2r_{i}}(BG_{(n)}, \mathbb{Z}(r_{i}))[t_{G}^{-1}]$ is equal to $p^{nr_{i}}.$
\item Let $$0 \neq \tilde{y} \in \mathbb{Z}[t_{G}^{-1}]\tilde{y}_{i_{1}} \oplus \cdots \oplus \mathbb{Z}[t_{G}^{-1}]\tilde{y}_{i_{M}}$$ where $|\tilde{y}_{i_{l}}|=r_{i}$ in $CH^{*}(BG_{(n)})[t_{G}^{-1}]$ for all $l.$ Assume that the additive order of $\tilde{y}$ is equal to $p^{nr_{i}}.$ Then there does not exist a relation of the form \begin{equation} \label{nontrivialrelation}p^{nr_{i}-1}\tilde{y}=p^{nr_{j_{1}}-1}f_{1}\tilde{y}_{j_{1}}+ \cdots +p^{nr_{j_{s}}-1}f_{s}\tilde{y}_{j_{s}}\end{equation} in $CH^{*}(BG_{(n)})[t_{G}^{-1}]$ where $r_{j_{l}}<r_{i}$ for all $l.$
\end{enumerate}
\end{proposition}
\begin{proof}
We assume that $k=\mathbb{F}_{p}$ throughout the proof. A base change argument implies the result for any field of characteristic $p.$ If $r_{i}=1,$ then we are done since $CH^{*}(B\mu_{p^{n}}) \cong \mathbb{Z}[x]/(p^{n}x)$ where $|x|=1.$ Assume $r_{i}>1.$ There is a map of fibrations 
\[
\begin{tikzcd}
G \arrow[d, ""] \arrow[r, ""] & G/G_{(n)} \cong G \arrow[d, ""] \\
G//T \arrow[d, ""] \arrow[r, ""] & (G/G_{(n)})//T  \arrow[d, ""] \\
BT \arrow[r, "id."] & BT
\end{tikzcd}
\]
where the horizontal maps are all flat (after choosing suitable approximations of $BT$). We consider the spectral sequences
$$E_{1}^{l,q}=H^{2l+q}_{T}(G, \mathbb{Z}( r_{i}+l))[t_{G}^{-1}] ^{\oplus \binom{m}{-l}} \Rightarrow H^{l+q}(G, \mathbb{Z}(r_{i}))[t_{G}^{-1}],$$
$$\overline{E}_{1}^{l,q}=H^{2l+q}_{T}(G/G_{(n)}, \mathbb{Z}( r_{i}+l))[t_{G}^{-1}] ^{\oplus \binom{m}{-l}} \Rightarrow H^{l+q}(G/G_{(n)}, \mathbb{Z}(r_{i}))[t_{G}^{-1}]$$ with differentials $d, d'$ respectively, from Theorem \ref{EilenbergMoore}. From the naturality of the Eilenberg-Moore spectral sequence, there is a map of spectral sequence $f:\overline{E} \to E$ that is compatible with flat pullback  $$H^{*}(G/G_{(n)}, \mathbb{Z}(*)) \to H^{*}(G, \mathbb{Z}(*)).$$ (The map $f$ can be defined at the level of cycles, using flat pullback.) On the $E_{1}$ page, $f$ is induced by pullback $$H^{*}_{T}(G/G_{(n)}, \mathbb{Z}(*)) \to H^{*}_{T}(G, \mathbb{Z}(*)).$$

\indent From \cite[Theorem 16.1]{Tot2}, $CH^{*}(BT)[t_{G}^{-1}]$ is free as a $CH^{*}(BG)[t_{G}^{-1}]$-module and $$H^{*}_{T}(G, \mathbb{Z}(*))[t_{G}^{-1}] \cong H^{*}(k, \mathbb{Z}(*)) \otimes _{CH^{*}(BG)} CH^{*}(BT) [t_{G}^{-1}],$$ $$H^{*}_{T}(G/G_{(n)}), \mathbb{Z}(*))[t_{G}^{-1}] \cong H^{*}(BG_{(n)}, \mathbb{Z}(*)) \otimes _{CH^{*}(BG)} CH^{*}(BT) [t_{G}^{-1}].$$ Let $CH^{*>0}(BG)CH^{*}(BT)[t_{G}^{-1}]$ denote the ideal of $CH^{*}(BT)[t_{G}^{-1}]$ generated by $CH^{*>0}(BG)[t_{G}^{-1}]$ (the kernel of the pullback map $CH^{*}(BT)[t_{G}^{-1}] \to H^{*}_{T}(G, \mathbb{Z}(*))[t_{G}^{-1}]$).

\indent Write $y_{i}=t_{1}p_{1}+\cdots +t_{m}p_{m}$ for some homogeneous $p_{1}, \ldots, p_{m} \in CH^{*}(BT)[t_{G}^{-1}].$ As $y_{i}$ is a generator of the polynomial ring $CH^{*}(BG)[t_{G}^{-1}]$ and $r_{i}>1,$ we can't have $p_{q} \in CH^{*>0}(BG)CH^{*}(BT)[t_{G}^{-1}]$ for all $q.$ Here, we are using that $CH^{*}(BG)[t_{G}^{-1}] \to CH^{*}(BT)[t_{G}^{-1}]$ admits a retraction \begin{equation} \label{retraction}r: CH^{*}(BT)[t_{G}^{-1}] \to CH^{*}(BG)[t_{G}^{-1}], \end{equation} preserving degree as a map of $CH^{*}(BG)[t_{G}^{-1}]$-modules. See the proof of \cite[Theorem 1.3]{Tot1} for an explicit description of this retraction. We let $$p_{1}, \ldots, p_{m} \in H^{2(r_{i}-1)}_{T}(G, \mathbb{Z}(r_{i}-1))[t_{G}^{-1}], H^{2(r_{i}-1)}_{T}(G/G_{(n)}, \mathbb{Z}(r_{i}-1))[t_{G}^{-1}] $$ denote the pullback of the elements $p_{1}, \ldots, p_{m} \in CH^{*}(BT)[t_{G}^{-1}].$

\indent Then we have a nonzero element \begin{equation} \label{definez} z=p_{1}dt_{1}+\cdots +p_{m}dt_{m} \in E_{1}^{-1, 2r_{i}}\end{equation} with $d_{1}(z)=0,$ and \begin{equation} \label{notnotinz} z \not \in d_{1}(E_{1}^{-2, 2r_{i}}) \end{equation} ($z$ is well-defined up to a boundary cycle). The element $z$ lifts to an element $\overline{z}=p_{1}dt_{1}+\cdots +p_{m}dt_{m} \in \overline{E}_{1}^{-1, 2r_{i}}$ where $d_{1}'(\overline{z})=\tilde{y}_{i} \neq 0$ by Lemma \ref{lemmachowBG}. Let $p^{a}$ denote the order of $\tilde{y}_{i}.$ Then $d_{1}'(p^{a}\overline{z})=0$ and $$p^{a}\overline{z} \not \in d_{1}'(\overline{E}_{1}^{-2, 2r_{i}})$$ by \ref{notnotinz}. As all differentials entering and leaving $E_{h}^{-1, 2r_{i}}$ are $0$ for $h>1,$ we have a nonzero $z \in E_{2}^{-1, 2r_{i}} \cong E_{\infty}^{-1, 2r_{i}}.$ Similarly, $p^{a}\overline{z}$ gives a nonzero element \begin{equation}\label{definetildez}\tilde{z} \in \overline{E}_{\infty}^{-1, 2r_{i}}.\end{equation} Then there exists $x \in H^{2r_{i}-1}(G/G_{(n)}, \mathbb{Z}(r_{i}))[t_{G}^{-1}]$ such that $x \mapsto \tilde{z}$ under the quotient map $$H^{2r_{i}-1}(G/G_{(n)}, \mathbb{Z}(r_{i}))[t_{G}^{-1}] \to \overline{E}_{\infty}^{-1, 2r_{i}} = H^{2r_{i}-1}(G/G_{(n)}, \mathbb{Z}(r_{i}))[t_{G}^{-1}]/\overline{E}_{\infty}^{0, 2r_{i}-1}.$$ From \cite[Theorem 4.6]{Gei}, $x \mapsto p^{nr_{i}}x$ under the map $$H^{2r_{i}-1}(G, \mathbb{Z}(r_{i}))[t_{G}^{-1}] \cong H^{2r_{i}-1}(G/G_{(n)}, \mathbb{Z}(r_{i}))[t_{G}^{-1}]\to H^{2r_{i}-1}(G, \mathbb{Z}(r_{i}))[t_{G}^{-1}].$$ We summarize our situation with the following commuting square.
\[
\begin{tikzcd}
H^{2r_{i}-1}(G/G_{(n)}, \mathbb{Z}(r_{i}))[t_{G}^{-1}] \arrow[d, "x \mapsto p^{nr_{i}}x"] \arrow[rr, "x \mapsto \tilde{z}"]&& \overline{E}_{\infty}^{-1, 2r_{i}} \cong  \overline{E}_{2}^{-1, 2r_{i}}\arrow[d, "\tilde{z} \mapsto p^{a}z"] \\ H^{2r_{i}-1}(G, \mathbb{Z}(r_{i}))[t_{G}^{-1}] \arrow[rr, "p^{nr_{i}}x \mapsto p^{a}z "] && E_{\infty}^{-1, 2r_{i}} \cong E_{2}^{-1, 2r_{i}}
\end{tikzcd}
\] Hence, $p^{a}z=p^{nr_{i}}\tilde{x}$ in $E_{2}^{-1, 2r_{i}}$ for some $\tilde{x} \in E_{2}^{-1, 2r_{i}}.$ Let $q_{1}dt_{1}+\cdots +q_{m}dt_{m}$ denote a lifting of $\tilde{x}$ to $E_{1}^{-1, 2r_{i}}$ where $q_{1}, \ldots, q_{m} \in CH^{*}(BT)[t_{G}^{-1}].$ We then have \begin{equation}\label{lastrelation} p^{a}y_{i}=p^{nr_{i}}(t_{1}q_{1}+\cdots +t_{m}q_{m})+(t_{1}g_{1}+\cdots +t_{m}g_{m})\end{equation} in $CH^{*}(BT)[t_{G}^{-1}]$ for some $$g_{1}, \ldots, g_{m} \in CH^{*>0}(BG)CH^{*}(BT)[t_{G}^{-1}]$$ where $$t_{1}q_{1}+\cdots +t_{m}q_{m} \in CH^{*>0}(BG)CH^{*}(BT)[t_{G}^{-1}] $$ as well. Applying the retraction \ref{retraction}, we get $$r(t_{1}g_{1}+\cdots +t_{m}g_{m}) \in \mathbb{Z}[t_{G}^{-1}][y_{1}, \ldots, y_{i-1}] \subset CH^{*}(BG)[t_{G}^{-1}] .$$ Therefore, $$r(p^{a}y_{i}-(p^{nr_{i}}(t_{1}q_{1}+\cdots +t_{m}q_{m})))=p^{a}y_{i}-p^{nr_{i}}r(t_{1}q_{1}+\cdots +t_{m}q_{m}) \in \mathbb{Z}[t_{G}^{-1}][y_{1}, \ldots, y_{i-1}]$$ which forces $p^{a}=p^{nr_{i}}.$ This proves the first part of the proposition.

\indent Now we prove the second part of the proposition. For simplicity, we assume that the fundamental degrees of $G$ are distinct and take $\tilde{y}=\tilde{y}_{i}.$ Suppose by way of contradiction that there exists such a relation. Consider the complex $$H^{*}(k, \mathbb{Z}(*))[t_{G}^{-1}] \otimes \bigwedge \nolimits^{*}(y_{1}, \ldots , y_{m})$$ (with the obvious differential) that computes \begin{equation} \label{flatbasechangeid}E_{2}^{l,q} \cong \textup{Tor}^{CH^{*}(BG)}_{-l,q,r_{i}}(\mathbb{Z}, H^{*}(k, \mathbb{Z}(*)))[t_{G}^{-1}].\end{equation} We have $$E_{2}^{-1,2r_{i}} \cong E_{\infty}^{-1,2r_{i}} \cong \mathbb{Z}[t_{G}^{-1}] \cdot dy_{i} \cong H^{2r_{i}-1}(G,\mathbb{Z}(r_{i}))[t_{G}^{-1}].$$ Under the identification \ref{flatbasechangeid}, the element $z$ from \ref{definez} gives a generator of $E_{2}^{-1,2r_{i}}.$

\indent Now consider the complex \begin{equation} \label{koszulbg_ntorsionlast}H^{*}(BG_{(n)}, \mathbb{Z}(*))[t_{G}^{-1}] \otimes \bigwedge \nolimits^{*}(y_{1}, \ldots , y_{m})\end{equation} (with the obvious differential $d$) that computes \begin{equation} \label{flatbasechangeidG_n}\overline{E}_{2}^{l,q} \cong \textup{Tor}^{CH^{*}(BG)}_{-l,q,r_{i}}(\mathbb{Z}, H^{*}(BG_{(n)}, \mathbb{Z}(*)))[t_{G}^{-1}].\end{equation} Under the map $f:\overline{E}_{\infty}^{-1,2r_{i}} \to E_{\infty}^{-1,2r_{i}}$ the element $\tilde{z}$ \ref{definetildez} maps to $p^{nr_{i}}$ times a generator of $E_{\infty}^{-1,2r_{i}}.$ It follows that $\tilde{z}$ is a generator of $$\overline{E}_{\infty}^{-1,2r_{i}} \cong H^{2r_{i}-1}(G/G_{(n)},\mathbb{Z}(r_{i}))[t_{G}^{-1}].$$ From the relation \ref{nontrivialrelation}, we get a nonzero $$x=-p^{nr_{1}-1}f_{1}dy_{1}- \cdots -p^{nr_{i-1}-1}f_{i-1}dy_{i-1}+p^{nr_{i}-1}dy_{i}$$ in the complex \ref{flatbasechangeidG_n} with $d(x)=0.$ The element $x$ then gives a nonzero $\tilde{x} \in \overline{E}_{\infty}^{-1,2r_{i}}.$ We have $\tilde{x}=a\tilde{z}$ for some $a \in \mathbb{Z}[t_{G}^{-1}].$ Therefore, $$-p^{nr_{1}-1}f_{1}dy_{1}- \cdots -p^{nr_{i-1}-1}f_{i-1}dy_{i-1}+(p^{nr_{i}-1}-ap^{nr_{i}})dy_{i}$$ is a boundary cycle in the complex \ref{koszulbg_ntorsionlast}, which is impossible since the constant $p^{nr_{i}-1}-ap^{nr_{i}}$ is not equal to $0.$
 \end{proof}

\indent From Proposition \ref{propositiontorsion}, for each $i \leq m,$ there exists $g_{i} \in H^{2r_{i}-1}(BG_{(n)}, \mathbb{Z}/p(r_{i})) $ such that $\delta(g_{i})=p^{nr_{i}-1}\tilde{y}_{i}.$ We can then consider $$\bigwedge \nolimits^{*}(g_{1}, \ldots , g_{m}) \subset H^{*}(BG_{(n)}, \mathbb{Z}/p(*)).$$

\begin{corollary}  \label{bocksteiniso} There is an isomorphism of $CH^{*}(BG)/p$-modules \begin{equation}H^{*}(BG_{(n)}, \mathbb{Z}/p(*)) \cong CH^{*}(BG)/p \otimes \bigwedge\nolimits ^{*}(g_{1}, \ldots , g_{m})\end{equation} where $$\bigwedge\nolimits ^{*}(g_{1}, \ldots , g_{m}) \cong H^{*}(G, \mathbb{Z}/p(*))$$ under the pullback \begin{equation} \label{pullbackbocksteiniso}H^{*}(BG_{(n)}, \mathbb{Z}/p(*)) \to H^{*}(G/G_{(n)}, \mathbb{Z}/p(*)).\end{equation}
\end{corollary}
\begin{proof} It suffices to show that pullback induces an isomorphism $$\bigwedge \nolimits ^{1}(g_{1}, \ldots , g_{m}) \cong M_{1}(G/G_{(n)}).$$ For simplicity, we assume that the fundamental degrees of $G$ are distinct. We proceed by induction on the fundamental degrees. From Theorem \ref{theoremMotcohBG_r}, recall that $$H^{*}(G/G_{(n)}, \mathbb{Z}/p(*)) \cong \mathbb{Z}/p \otimes_{CH^{*}(BG)/p}H^{*}(BG_{(n)}, \mathbb{Z}/p(*)).$$ For the base case, $g_{1}$ generates $H^{2r_{1}-1}(BG_{(n)}, \mathbb{Z}/p(r_{1}))$ which implies the result. Now assume that $\bigwedge \nolimits ^{1}(g_{1}, \ldots , g_{i-1})$ injects into $M_{1}(G/G_{(n)}).$ Suppose by way of contradiction that $g_{i}$ maps to $0$ under the pullback \ref{pullbackbocksteiniso}. Then $$g_{i}=f_{1}g_{1}+ \cdots +f_{i-1}g_{i-1}$$ for some $f_{1}, \cdots, f_{i-1} \in CH^{*}(BG)/p,$ not all zero. Applying the integral Bockstein $\delta$ to $g_{i},$ we get $$p^{nr_{i}-1}\tilde{y}_{i}=p^{nr_{1}-1}f_{1}\tilde{y}_{1}+ \cdots +p^{nr_{i-1}-1}f_{i-1}\tilde{y}_{i-1} \in CH^{*}(BG_{(n)})[t_{G}^{-1}],$$ contradicting Proposition \ref{propositiontorsion}.
\end{proof}

\begin{remark} Theorem \ref{theoremMotcohBG_r} and Proposition \ref{propositiontorsion} give information about the integral motivic cohomology of Frobenius kernels. For example, consider $G=GL_{r}$ for $r \in \mathbb{N}.$ We have $CH^{*}(BGL_{r}) \cong \mathbb{Z}[c_{1}, \ldots, c_{r}].$ From the Eilenberg-Moore spectral sequence (using integral motivic cohomology since $t_{GL_{r}}=1$), $CH^{*}(BGL_{r})$ maps surjectively onto $CH^{*}(BG_{(n)})$ and $p^{ni}c_{i}=0$ in $CH^{*}(BG_{(n)})$ for all $i.$

\indent As noted above, for each $i \leq r,$ there exists $g_{i} \in H^{2i-1}(BG_{(n)}, \mathbb{Z}/p(i))$ such that $\delta(g_{i})=p^{ni-1}c_{i}.$ Therefore, $$\delta(g_{i}c_{j})=p^{ni-1}c_{i}c_{j}=0 \in H^{*}(BG_{(n)}, \mathbb{Z}(*))$$ whenever $i>j.$ It follows that $g_{i}c_{j}$ lifts to a nonzero element in $H^{*}(BG_{(n)}, \mathbb{Z}(*))$ for all $i>j.$ It doesn't seem as though $H^{*}(BG_{(n)}, \mathbb{Z}(*))$ has a clean description.
\end{remark}
\section{Hodge-Witt sheaves and cohomology}
\indent  Fix a perfect ground field $k$ of characteristic $p>0.$ Let $X/\mathbb{F}_{p}$ be a scheme. For $r \geq 1,$ we consider the de Rham-Witt complex $W_{r} \Omega^{*}$ of $X$ as constructed in \cite{Ill}. Each $W_{r} \Omega^{m}$ is a $W_{r}(\mathcal{O}_{X})$-module.

\indent We now recall the related construction of logarithmic Hodge-Witt sheaves. We refer to \cite[\S 4]{Mor} for some properties of logarithmic Hodge-Witt sheaves that hold for arbitrary $\mathbb{F}_{p}$-schemes. Define $\textup{dlog}: \mathcal{O}^{*}_{X} \to W_{r}\Omega^{1}$ by $$\textup{dlog}(x)=\frac{d \underline{x}}{\underline{x}}$$ where $\underline{x} \coloneqq (x,0, \ldots)$ is the Teichm\"uller representative of $x.$ We then define the logarithmic Hodge-Witt sheaf $$W_{r} \Omega_{\textup{log}}^{m} \subset W_{r} \Omega^{m}$$ to be the subsheaf generated \'etale locally by forms $\textup{dlog}(a_{1}) \wedge \cdots \wedge \textup{dlog}(a_{m}).$ Each $W_{r} \Omega_{\textup{log}}^{m}$ is a sheaf of $\mathbb{Z}/p^{r}$-modules in the \'etale or Zariski topology.

\indent Now assume that $X/k$ is smooth. Let $D(X_{\tau})$ denote the derived category of sheaves in the topology $\tau \in \{\textup{Zar}, \textup{\'et} \}.$ Let $\mathbb{Z}/p^{r}(j) \in D(X_{\tau})$ denote the object representing weight $j$ motivic cohomology with $\mathbb{Z}/p^{r}$-coefficients. (Representatives are given by Bloch's cycle complexes.) Geisser-Levine proved that $$\mathbb{Z}/p^{r}(j) \cong W_{r} \Omega_{\textup{log}}^{j}[-j]$$ in $D(X_{\tau})$ \cite[Theorem 8.3]{GeiLev}. Hence, there are isomorphisms $$H^{i}_{\tau}(X, \mathbb{Z}/p^{r}(j)) \cong H^{i-j}_{\tau}(X, W_{r} \Omega_{\textup{log}}^{j})$$ for all $i,j.$ These isomorphisms are compatible with product structures.

\section{Cycle class map} \label{sectioncycleclass}

\indent Let $X$ be a smooth algebraic stack over a perfect field $k$ and let $F$ be an abelian sheaf on the big \'etale site of $X.$ The cohomology groups of $X$ can be computed using any smooth presentation of $X.$ Let $U \to X$ be a smooth surjective morphism where $U$ is an algebraic space. Consider the \textit{\v{C}ech nerve} $C(U/X):$ 
\[
\begin{tikzcd}
  U \arrow[r, bend right=50, ""] & U \times _{X} U \arrow[l, yshift=.7ex,  ""] \arrow[l, yshift=.-.7ex,  ""]  \arrow[r, bend right=50, yshift=.7ex, ""] \arrow[r, bend right=50, yshift=-.7ex, ""] & U\times_{X} U \times _{X} U \cdots  \arrow[l, yshift=-.8ex,  ""] \arrow[l,   ""] \arrow[l, yshift=.8ex,  ""]
\end{tikzcd}
\]
associated to $U \to X.$ Then $$H^{*}_{\textup{\'et}}(X, F) \cong H^{*}_{\textup{\'et}}(C(U/X), F)$$ \cite[Tags 0DGB, 06XJ]{Sta}. In particular, there is a spectral sequence $$E_{1}^{i,j}=H^{j}_{\textup{\'et}}(U_{X}^{i+1}, F) \rightarrow H^{i+j}_{\textup{\'et}}(X, F).$$ A similar setup was used by Totaro in \cite{Tot} to study the Hodge cohomology of classifying stacks.

\indent Now we specialize to positive characteristic. Let $k$ be a perfect field of characteristic $p>0$ and let $X$ be a smooth stack over $k.$ We consider the abelian sheaf $W_{r}\Omega^{j}_{\textup{log}}$ of logarithmic Hodge-Witt forms over $k$ on the big \'etale site of $X$ and define $$H^{i}_{\textup{\'et}}(X, \mathbb{Z}/p^{r}(j)) \coloneqq H^{i-j}_{\textup{\'et}}(X, W_{r}\Omega^{j}_{\textup{log}}).$$ From the inclusion $\Omega^{j}_{\textup{log}} \subset \Omega^{j},$ we get a map into the Hodge cohomology of $X:$ $$H^{i}_{\textup{\'et}}(X, \mathbb{Z}/p(j)) \to  H^{i-j}_{\textup{\'et}}(X, \Omega^{j}).$$ For $i \geq 0,$ we set $$H^{i}_{\textup{H}}(X/k)\coloneqq\bigoplus_{j=0}^{\infty}H^{i-j}_{\textup{\'et}}(X, \Omega^{j}).$$

\indent Let $G$ be a flat affine group scheme of finite type over $k$ and let $\mathcal{B}G$ denote the classifying stack of $G.$ Note that $\mathcal{B}G$ is smooth over $k$ even if $G$ is not smooth \cite[Tag 0DLS]{Sta}. We now define a cycle map from the mod $p^{r}$ motivic cohomology of the classifying space $BG$ to the mod $p^{r}$ \'etale motivic cohomology of the classifying stack $\mathcal{B}G.$ Let $i,j \geq 0.$ Pick a representation $G \to GL(V)$ where $G$ acts freely on some $U \coloneqq V \setminus S$ where $S$ is a closed $G$-invariant subset of $V$ of codimension $>j,$ and such that $U/G$ exists as a smooth scheme over $k.$ The $G$-torsor $U \to U/G$ induces a morphism $U/G \to \mathcal{B}G$ that is smooth and surjective since $U$ is smooth over $k.$ The associated \v{C}ech nerve is given by \[
\begin{tikzcd}
  U/G \arrow[r, bend right=50, ""] & U^{2}/G \arrow[l, yshift=.7ex,  ""] \arrow[l, yshift=.-.7ex,  ""]  \arrow[r, bend right=50, yshift=.7ex, ""] \arrow[r, bend right=50, yshift=-.7ex, ""] & U^{3}/G \cdots  \arrow[l, yshift=-.8ex,  ""] \arrow[l,   ""] \arrow[l, yshift=.8ex,  ""]
\end{tikzcd}
\] As motivic cohomology is homotopy invariant, the cosimplicial group 
\[
\begin{tikzcd}
  H^{i-j}_{\textup{Zar}}(U/G, W_{r}\Omega_{\textup{log}}^{j})  \arrow[r, yshift=.7ex,  ""] \arrow[r, yshift=.-.7ex,  ""] & H^{i-j}_{\textup{Zar}}(U^{2}/G, W_{r}\Omega_{\textup{log}}^{j})  \arrow[r, yshift=-.8ex,  ""] \arrow[l, bend right=-50, ""] \arrow[r,   ""] \arrow[r, yshift=.8ex,  ""]  & H^{i-j}_{\textup{Zar}}(U^{3}/G, W_{r}\Omega_{\textup{log}}^{j}) \cdots \arrow[l, bend right=-50, yshift=.7ex, ""] \arrow[l, bend right=-50, yshift=-.7ex, ""]
\end{tikzcd}
\]
is constant (meaning that the coface maps are all isomorphisms) by \cite[Lemma A.4]{BliMer}. Then $$H^{i-j}_{\textup{Zar}}(C((U/G)/\mathcal{B}G), W_{r}\Omega_{\textup{log}}^{j})\cong H^{i-j}_{\textup{Zar}}(U/G, W_{r}\Omega_{\textup{log}}^{j}).$$ The canonical map of sites $C((U/G)/\mathcal{B}G)_{\textup{\'et}} \to C((U/G)/\mathcal{B}G)_{\textup{Zar}}$ gives a map of topoi $$(\epsilon^{*}, \epsilon_{*}): Sh(C((U/G)/\mathcal{B}G)_{\textup{Zar}}) \to Sh(C((U/G)/\mathcal{B}G)_{\textup{\'et}})$$ where $\epsilon^{*}$ is also exact. (This is induced by the usual map of topoi $$(\epsilon^{*}, \epsilon_{*}): Sh((U^{m}/G)_{\textup{Zar}}) \to Sh((U^{m}/G)_{\textup{\'et}})$$ for each $m \geq 0$ where $\epsilon^{*}$ is \'etale sheafification and $\epsilon_{*}$ is restriction \cite[Tag 0D95, Tag 0D96]{Sta}. See \cite[Tag 09VK]{Sta} for the definition of the Zariski site of a simplicial scheme. The \'etale site is defined similarly.) We then get a comparison map
\begin{equation} \label{comparisonmap}
H^{i}(BG, \mathbb{Z}/p^{r}(j)) \cong H^{i-j}_{\textup{Zar}}(U/G, W_{r}\Omega_{\textup{log}}^{j}) \cong H^{i-j}_{\textup{Zar}}(C((U/G)/\mathcal{B}G), W_{r}\Omega_{\textup{log}}^{j})\end{equation}

$$ \to H^{i-j}_{\textup{\'et}}(C((U/G)/\mathcal{B}G), W_{r}\Omega_{\textup{log}}^{j}) = H^{i}_{\textup{\'et}}(\mathcal{B}G, \mathbb{Z}/p^{r}(j)).$$
\begin{proposition} The comparison map \ref{comparisonmap} is well-defined.
\end{proposition}
\begin{proof}
Essentially, we just use that the motivic classifying space $BG$ is well-defined and that the \'etale motivic cohomology of the classifying stack $\mathcal{B}G$ can be computed using any smooth presentation of $\mathcal{B}G.$ Let $S' \subset V$ be a closed $G$-invariant subset containing $S,$ and of codimension $>j.$ Set $U'=V \setminus S'.$ Then $U'/G \to \mathcal{B}G$ is a smooth presentation and the map of simplicial schemes $C((U'/G)/\mathcal{B}G)\to C((U/G)/\mathcal{B}G)$ induces isomorphisms $$H^{i-j}_{\textup{\'et}}(C((U/G)/\mathcal{B}G), W_{r}\Omega_{\textup{log}}^{j}) \cong H^{i-j}_{\textup{\'et}}(\mathcal{B}G, W_{r}\Omega_{\textup{log}}^{j}) \cong H^{i-j}_{\textup{\'et}}(C((U'/G)/\mathcal{B}G), W_{r}\Omega_{\textup{log}}^{j}).$$ As $$H^{i}(U/G, \mathbb{Z}/p^{r}(j)) \cong H^{i}(U'/G, \mathbb{Z}/p^{r}(j)),$$ this shows that the comparison map does not depend on the choice of $S \subset V.$

\indent Next, we show that the comparison map does not depend on the choice of $V.$ Let $G \to GL(V')$ be a representation, and let $S' \subset V'$ be a $G$-invariant closed subset of codimension $>j.$ We have $$H^{i-j}_{\textup{\'et}}(C(((U\times V')/G)/\mathcal{B}G), W_{r}\Omega_{\textup{log}}^{j}) \cong H^{i-j}_{\textup{\'et}}(\mathcal{B}G, W_{r}\Omega_{\textup{log}}^{j}) \cong H^{i-j}_{\textup{\'et}}(C((U/G)/\mathcal{B}G), W_{r}\Omega_{\textup{log}}^{j})$$ and 
$$H^{i-j}_{\textup{\'et}}(C(((V\times U')/G)/\mathcal{B}G), W_{r}\Omega_{\textup{log}}^{j}) \cong H^{i-j}_{\textup{\'et}}(\mathcal{B}G, W_{r}\Omega_{\textup{log}}^{j}) \cong H^{i-j}_{\textup{\'et}}(C((U'/G)/\mathcal{B}G), W_{r}\Omega_{\textup{log}}^{j}).$$ There are also isomorphisms on the corresponding motivic cohomology groups. As we showed in the previous paragraph that the comparison map doesn't depend on the choice of $G$-invariant closed subset in the representation $V \oplus V',$ this shows that the comparison map doesn't depend on the choice of representation $V.$
\end{proof}

\indent We note that our cycle class maps are compatible with product structures. Now we look at some examples. The following result was proved by Srinivas \cite[section 3]{Sri}.

\begin{proposition} \label{propcyclemap} Let $G$ be a split reductive field over a field $F$ and let $P \subset G$ be a parabolic subgroup. Then the cycle class map $$CH^{*}(G/P) \otimes F \to H^{*}_{\textup{H}}((G/P)/k)$$ is an isomorphism.
\end{proposition}
\indent Fix a perfect field $k$ of characteristic $p>0.$
\begin{corollary} \label{cortorus} Let $T$ be a split torus over $k.$ Then the cycle class map defined above induces an isomorphism $$H^{*}(BT, \mathbb{Z}/p(*)) \otimes k \cong H^{*}_{\textup{H}}(\mathcal{B}T/k).$$
\end{corollary}
\begin{proof} From \cite{Tot}, we have $$H^{*}_{\textup{H}}(\mathcal{B}\mathbb{G}_{m}/k) \cong k[c_{1}]$$ where $c_{1} \in H^{1}_{\textup{\'et}}(\mathcal{B}\mathbb{G}_{m}, \Omega^{1}).$ From Proposition \ref{propcyclemap}, $$CH^{*}(\mathbb{P}^{n}) \otimes k \cong H^{*}_{\textup{H}}(\mathbb{P}^{n}/k)$$ for all $n.$ It follows that the cycle class map induces an isomorphism $$H^{*}(B\mathbb{G}_{m}, \mathbb{Z}/p(*)) \otimes k \cong H^{*}_{\textup{H}}(\mathcal{B}\mathbb{G}_{m}/k).$$ The K\"unneth formula for the motivic cohomology of $BT$ and the Hodge cohomology of $\mathcal{B}T$ \cite[Proposition 5.1]{Tot} then give the general result.
\end{proof}

\begin{corollary}  Let $G$ be a split reductive group over $k$ with $\textup{char}(k)=p$ a non-torsion prime for $G.$ Then the cycle class map induces an isomorphism $$CH^{*}(BG) \otimes k \cong H^{*}_{\textup{H}}(\mathcal{B}G/k).$$
\end{corollary}
\begin{proof} 
As $p$ is a non-torsion prime, the restriction map $$CH^{*}(BG)/p \to CH^{*}(BT)/p$$ is an injection where $T \subset G$ is a split torus. We have $$H^{*}_{\textup{H}}(\mathcal{B}G/k) \cong \bigoplus _{i=0}^{\infty} H^{i}_{\textup{\'et}}(\mathcal{B}G, \Omega^{i}) \cong \mathcal{O}(\frak{g})^{G}$$ and $H^{*}_{\textup{H}}(\mathcal{B}G/k)$ is a polynomial ring with generators of degree equal to twice the fundamental degrees of $G$  \cite[Theorem 9.1, Theorem 9.2]{Tot}.] As noted in the proof of \cite[Lemma 8.2]{Tot}, the restriction map $$H^{*}_{\textup{H}}(\mathcal{B}G/k) \cong \mathcal{O}(\frak{g})^{G} \to \mathcal{O}(\frak{t}) \cong H^{*}_{\textup{H}}(\mathcal{B}T/k)$$ is injective. The result then follows from Corollary \ref{cortorus}.
\end{proof}

\indent We next look at the cycle map for Frobenius kernels and Chow groups.

\begin{proposition} \label{comparisonmup^n} Let $n \in \mathbb{N}.$The cycle class map induces an injection $$CH^{*}(B\mu_{p^{n}})/p  \to H^{*}_{\textup{\'et}}(\mathcal{B}\mu_{p^{n}}, \mathbb{Z}/p(*)).$$
\end{proposition}
\begin{proof} Take $$H^{*}(B\mu_{p^{n}}, \mathbb{Z}/p(*)) \cong \mathbb{F}_{p}[u,v]/(u^{2})$$ as in Proposition \ref{CohBmu}. A calculation similar to the one performed in the proof of \cite[Proposition 10.1]{Tot} shows that $$H^{*}_{\textup{H}}(\mathcal{B}\mu_{p^{n}}/k) \cong k[c_{1}]\langle w\rangle$$ (the free exterior algebra over $k[c_{1}]$ with generator $w$) where $c_{1} \in H^{1}_{\textup{\'et}}(\mathcal{B}\mu_{p^{n}}, \Omega^{1})$ is the pullback of a generator of $H^{*}_{\textup{H}}(\mathcal{B}\mathbb{G}_{m}/k),$ and $w \in H^{0}_{\textup{\'et}}(\mathcal{B}\mu_{p^{n}}, \Omega^{1})$ with $w^{2}=0.$ From Corollary \ref{cortorus}, it follows that $CH^{*}(B\mu_{p^{n}})/p$ injects into $H^{2*}_{\textup{\'et}}(\mathcal{B}\mu_{p^{n}}, \mathbb{Z}/p(*))$ under the cycle class map.
\end{proof}

\begin{corollary} \label{chowcomparison}
Let $G$ be a split reductive group over a perfect field $k$ with $p=\textup{char}(k)$ a non-torsion prime for $G.$ Let $n \in \mathbb{N}.$ Then the cycle map induces an injection $$CH^{*}(BG_{(n)})/p \to H^{*}_{\textup{\'et}}(\mathcal{B}G_{(n)}, \mathbb{Z}/p(*)).$$
\end{corollary}
\begin{proof}
From the proof of Lemma \ref{lemmachowBG}, the pullback map $$CH^{*}(BG_{(n)})/p \to CH^{*}(BT_{(n)})/p$$ is injective. The result then follows from Proposition \ref{comparisonmup^n}.
\end{proof}

\end{document}